\documentclass[11pt]{amsart}
\usepackage{amsmath,amssymb,latexsym,esint,cite,mathrsfs}
\usepackage{verbatim,wasysym,cite}
\usepackage[left=2.7cm,right=2.7cm,top=2.7cm,bottom=2.7cm]{geometry}
\usepackage[latin1]{inputenc}
\usepackage{microtype}
\usepackage{color,enumitem,graphicx}
\usepackage[colorlinks=true,urlcolor=blue, citecolor=red,linkcolor=blue,
linktocpage,pdfpagelabels, bookmarksnumbered,bookmarksopen]{hyperref}
\usepackage[hyperpageref]{backref}
\usepackage[english]{babel}

\renewcommand{\epsilon}{{\varepsilon}}

\numberwithin{equation}{section}
\newtheorem{theorem}{Theorem}[section]
\newtheorem{lemma}[theorem]{Lemma}
\newtheorem{remark}[theorem]{Remark}
\newtheorem{open}[theorem]{Open Problem}

\newtheorem{proposition}[theorem]{Proposition}

\title[Soliton dynamics for logarithmic NLS]
{Gausson dynamics for logarithmic \\ Schr\"odinger equations}

\author[A.H. Ardila]{Alex H. Ardila}
\author[M.\ Squassina]{Marco Squassina}

\address[A.H. Ardila]{Instituto Nacional de Matem\'atica Pura e Aplicada - IMPA
	\newline\indent
Estrada Dona Castorina, 110 CEP 22460-320
	\newline\indent
Rio de Janeiro, RJ - Brasil}
\email{ardila@impa.br}

\address[M.\ Squassina]{Dipartimento di Matematica e Fisica \newline\indent
	Universit\`a Cattolica del Sacro Cuore \newline\indent
	Via dei Musei 41, I-25121 Brescia, Italy}
\email{marco.squassina@unicatt.it}


\thanks{The first author was partially supported by CNPq/Brazil, through grant No. 152672/2016-8. The second author is member of the Gruppo Nazionale per l'Analisi Matematica, la Probabilit\`a e le loro Applicazioni and of INdAM}

\begin{document}

\begin{abstract}
	In this paper we study the validity of a Gausson (soliton) dynamics of the logarithmic Schr\"odinger equation in presence of a smooth external potential.
\end{abstract}

\maketitle

%


\section{Introduction}
This paper is concerned with the so-called soliton dynamics behaviour for the logarithmic Schr\"odinger 
equation with an external potential
\begin{equation}\label{0NL}
\begin{cases} 
i\epsilon\partial_{t}u_{\epsilon}+\frac{\epsilon^{2}}{2}\Delta u_{\epsilon}-V(x)u_{\epsilon}+\,u_{\epsilon}\,\mathrm{Log}\,\left|u_{\epsilon}\right|^{2}=0\\
u_{\epsilon}(0,t)=u_{\epsilon,0}(x),
\end{cases} 
\end{equation}
that is the study of the behaviour of the solution $u_{\epsilon}$,  in the semi-classical limit $\epsilon\rightarrow 0$;
namely when the Planck constant $\epsilon=\hbar$ tends to zero, by taking as initial datum for the Cauchy problem \eqref{0NL} a function ({\em Gausson type}) of the form
\begin{equation}\label{ID}
u_{\epsilon,0}(x)=e^{\frac{i}{\epsilon}x\cdot v_{0}}R\Big(\frac{x-x_{0}}{\epsilon}\Big), \quad R(x):=e^{\frac{1+N}{2}}e^{-|x|^{2}}, \,\, x\in \mathbb{R}^{N}.
\end{equation}
Here,  $u_{\epsilon}=u_{\epsilon}(x,t)$ is a complex-valued function of $(x,t)\in \mathbb{R}^{N}\times\mathbb{R}$, $N\geq1$, $i$ is the imaginary unit, 
$V:{\mathbb{R}^{N}}\rightarrow \mathbb{R}$ is an external potential and $x_{0}$, $v_{0}\in \mathbb{R}^{N}$, $v_{0}\neq 0$, are the initial position and velocity for the Newtonian system
\begin{equation}\label{HS}
\begin{cases} 
 \dot{x}(t)=\nu(t), \quad x(0)=x_{0}\\
 \dot{\nu}(t)=-\nabla V(x(t)), \quad \nu(0)=v_{0}.
\end{cases} 
\end{equation}
Notice that the classical Hamiltonian related to \eqref{HS} is 
\begin{equation}\label{H}
\mathcal{H}(t)=\frac{1}{2}|\nu(t)|^{2}+V(x(t))
\end{equation} 
and is conserved in the time.

Equation~\eqref{0NL} was proposed by Bialynicki-Birula and Mycielski \cite{CAS} in 1976 as a model of nonlinear wave mechanics. This NLS equation has wide applications in quantum optics \cite{Ap1}, nuclear physics \cite{HE},  geophysical applications of magma transport \cite{MFG}, effective quantum and gravity, theory of superfluidity,  Bose-Einstein condensation and open quantum systems; see \cite{OATL, APLES} and the references therein.  We refer to \cite{CL, TA, CALO, AHA1, FSCALE} for a study of existence and stability of standing waves, as well as for a study of the Cauchy problem in a suitable functional framework.

Rigorous results about of the soliton dynamics for nonlinear Schr\"odinguer equation with a power nonlinearity $|u|^{p-1}u$ were obtained in various papers by J.C. Bronski, R.L. Jerrard \cite{BJ} and S. Keraani \cite{SK2}. The main ingredients of the argument are the {\em conservation laws} of NLS and of the Hamiltonian \eqref{H} combined with {\em modulational stability estimates} proved by M. Weinstein \cite{ELW, MN5}. 
In recent years, the so-called soliton dynamics has attracted a great deal of attention from both the mathematicians and physicists; see for example \cite{SQ1, AS1, SSQ1, ADSQ1, EMBPMS}.  

Throughout this paper we assume that the potential $V$ in \eqref{0NL} is a $C^{3}({\mathbb R}^{N})$ function bounded with its derivatives.
Formally, the NLS \eqref{0NL} has the following two conserved quantities. The first conserved quantity is the
energy $E_{\epsilon}$ defined by
\begin{equation*}
E_{\epsilon}(u):=\frac{1}{2\epsilon^{N-2}}\int_{\mathbb{R}^{N}}\left|\nabla u\right|^{2}dx+\frac{1}{\epsilon^{N}}\int_{\mathbb{R}^{N}}V(x)|u|^{2}dx-\frac{1}{\epsilon^{N}}\int_{\mathbb{R}^{N}}\left|u\right|^{2}\mbox{Log}\left|u\right|^{2}dx.
\end{equation*}
The second conserved quantity is the mass,
\begin{equation*}
Q_{\epsilon}(u):=\frac{1}{\epsilon^{N}}\int_{\mathbb{R}^{N}}|u|^{2}\,dx.
\end{equation*}
Notice that due to the singularity of the logarithm at the origin, the energy {\em fails} to be finite as well of class $C^{1}$ on $H^{1}(\mathbb{R}^{N})$. Therefore, we consider the reflexive Banach space 
\begin{equation}\label{ASE}
W(\mathbb{R}^{N}):=\left\{u\in H^{1}(\mathbb{R}^{N}):\left|u\right|^{2}\mathrm{Log}\left|u\right|^{2}\in L^{1}(\mathbb{R}^{N})\right\}.
\end{equation}
It is well known that the energy $E_{\epsilon}$ is well-defined and of class $\mathcal{C}^1$ on $W({\mathbb R}^{N})$ (see Section \ref{S:0}). Notice that if $u\in C(\mathbb{R},{W}({\mathbb{R}^{N}}))\cap C^{1}(\mathbb{R}, {W}^{\prime}({\mathbb{R}^{N}}))$, then equation \eqref{0NL} makes sense in ${W}^{\prime}({\mathbb{R}^{N}})$,  where ${W}^{\prime}({\mathbb{R}^{N}})$ is the dual space of ${W}({\mathbb{R}^{N}})$. 

We see that the well-posedness of the Cauchy Problem for \eqref{0NL} in ${W}^{\prime}({\mathbb{R}^{N}})$ and the conservation laws follow by a standard compactness method developed in \cite[Chapter 9]{CB}.

\begin{proposition} \label{PCS}
Let $\epsilon>0$. For every $u_{\epsilon, 0}\in {W}({\mathbb{R}}^{N})$, there is a unique  global solution  $u_{\epsilon}\in C(\mathbb{R},{W}({\mathbb{R}^{N}}))\cap C^{1}(\mathbb{R}, {W}^{\prime}({\mathbb{R}^{N}}))$  of Eq. \eqref{0NL}  such that $u_{\epsilon}(x, 0)=u_{\epsilon, 0}$ and $\sup_{t\in \mathbb{R}}\left\|u_{{\epsilon}}(t)\right\|_{{W}({\mathbb{R}^{N}})}<\infty$. Furthermore, the solution $u_{{\epsilon}}(t)$ satisfies the conservation laws: 
\begin{equation*}
E_{\epsilon}(u_{\epsilon}(t))=E(u_{\epsilon, 0})\quad  and \quad Q_{\epsilon}(u_{\epsilon}(t))=Q_{\epsilon}(u_{\epsilon, 0})\quad  \text{for all $t\in \mathbb{R}$}.
\end{equation*}
\end{proposition}

We denote 
$$
\Sigma(\mathbb{R}^{N}):=\left\{u\in H^{1}(\mathbb{R}^{N}):|x|^{}u\in L^{2}(\mathbb{R}^{N}) \right\}.
$$ 
It is well known that $\Sigma(\mathbb{R}^{N})$ is a Hilbert space when is equipped with the norm 
\begin{equation*}
\|u\|_{\Sigma(\mathbb{R}^{N})}:=\sqrt{\int_{\mathbb{R}^{N}}\left(\left|\nabla u\right|^{2}+|x|^{2}|u|^{2}+|u|^{2}\right)dx},
\end{equation*}
and it is continuously embedded in $H^{1}(\mathbb{R}^{N})$. From \cite[Theorem 1.5]{RCIG} we have that if  initial data $u_{\epsilon, 0}$ belong to $\Sigma(\mathbb{R}^{N})$, then the solution $u_{\epsilon}(x,t)$ of Eq. \eqref{0NL} belong to $L^{\infty}_{\rm loc}(\mathbb{R}, \Sigma(\mathbb{R}^{N}))$. Moreover, if $u_{\epsilon, 0}\in H^{2}(\mathbb{R}^{N})$, then  $u_{\epsilon}(t)\in H^{2}(\mathbb{R}^{N})$ and $\partial_{t}u_{\epsilon}(t)\in L^{2}(\mathbb{R}^{N})$, for all $t\in \mathbb{R}$. 

Notice that the initial data $u_{\epsilon,0}$ in \eqref{ID} belong to $\Sigma(\mathbb{R}^{N})\cap H^{2}(\mathbb{R}^{N})$. On the other hand,  it is not hard to show that one has the following chain of continuous embeddings (see Lemma \ref{POii} below)
$$
\Sigma(\mathbb{R}^{N})\hookrightarrow W(\mathbb{R}^{N})\hookrightarrow H^{1}(\mathbb{R}^{N}). 
$$
In particular, since $E_{\epsilon}$ is of class $C^1$ on $W({\mathbb R}^{N})$, it follows that $E_{\epsilon}$ is of class $C^1$ on $\Sigma({\mathbb R}^{N})$.

Let $\omega\in \mathbb{R}$ and $\varphi\in {W}^{}({{\mathbb{R}^{N}}})$ be solution of the semilinear elliptic equation 
\begin{equation}\label{EP}
-\frac{1}{2}\Delta \varphi+\omega \varphi-\varphi\, \mathrm{Log}\left|\varphi \right|^{2}=0, \quad x\in {\mathbb{R}}^{N},
\end{equation}
It is well known that the Gausson 
$$
\phi_{\omega}(x):=e^{\frac{\omega-1}{2}}R(x), 
$$
where the function $R(x)$ is defined in \eqref{ID},  solves the problem \eqref{EP} for any dimension $N$. Furthermore, $\phi_{\omega}(x)$ is the unique, up to translations,  strictly positive $C^{2}$-solution for \eqref{EP} such that $\phi_{\omega}(x)\rightarrow 0$ as $\left|x\right|\rightarrow \infty$; see \cite[Theorem 1.2]{ADSQ1}. 

Orbital stability of Gaussons solutions $\phi_{\omega}(x)$ have been studied in various settings. More specifically,  Cazenave \cite{CL}; Cazenave and Lions \cite{CALO}; Ardila \cite{AHA1}; Blanchard and co. \cite{PHBJ, PHST}; research  the orbital stability of stationary solutions of \eqref{0NL}. 

As mentioned above,  the modulational stability property of ground states plays an important role in soliton dynamics; however, due to the singularity of the logarithm at the origin, {\em it is not clear whether the energy functional is of class $C^{2}$} in a tubular neighbourhood of the Gausson $R$.  In particular, it is an open problem to determine whether the Gausson $R$ satisfies the modulational stability estimates. 

Consider $H^{1}(\mathbb{R}^{N})$ equipped with the scaled norm  
$$
\|\phi\|^{}_{H_{\epsilon}^{1}}=\sqrt{\epsilon^{2-N}\|\nabla\phi\|^{2}_{L^{2}}
+\epsilon^{-N}\|\phi\|^{2}_{L^{2}}}. 
$$
The following result is obtained by using the only information that the minimizing sequences for the
constrained variational problem associated to \eqref{EP} are precompact in $W({\mathbb R}^{N})$.

\begin{theorem} \label{T1}
Let $u_{{\epsilon}}\in \Sigma(\mathbb{R}^{N})$ be the family  of solutions to the Cauchy problem \eqref{0NL} with initial data \eqref{ID}, for some $x_{0}$, $v_{0}\in \mathbb{R}^{N}$. Then there exist a positive constant $C$, independent of $\epsilon>0$, such that 
$$
\sup_{t\in \mathbb{R}}\|\nabla u_{{\epsilon}}(t)\|^{2}_{L^{2}}\leq C\epsilon^{N-2}. 
$$
Moreover, for any $\eta>0$ there exist $\epsilon>0$, a time $T^{\ast}_{\epsilon}>0$ and continuous functions 
$$
\theta_{\epsilon}: [0, T^{\ast}_{\epsilon}]\rightarrow [0, 2\pi],\quad y_{\epsilon}:[0, T^{\ast}_{\epsilon}]\rightarrow \mathbb{R}^{N},
$$ 
such that 
\begin{equation}\label{Ap}
u_{{\epsilon}}(x,t)=e^{\frac{i}{\epsilon}(\nu(t)\cdot x+\theta_{\epsilon}(t))}e^{\frac{1+N}{2}}e^{-\frac{1}{\epsilon^{2}}|{x-y_{\epsilon}(t)}|^{2}}+\omega_{\epsilon}(x,t),
\end{equation}
where $\|\omega_{\epsilon}(t)\|^{}_{H_{\epsilon}^{1}}<\eta^{2}$ for all $t\in [0, T^{\ast}_{\epsilon})$. 

Here $y_{\epsilon}=x(t)+\epsilon z_{\epsilon}(t)$ for some continuous function $z_{\epsilon}:[0, T^{\ast}_{\epsilon}]\rightarrow \mathbb{R}^{N}$, where $(x(t), \nu(t))$ is the solution of the classical Hamiltonian system \eqref{HS}.
\end{theorem}

As it is well known, to prove the  modulational stability property of ground states,  it is necessary to study the spectral structure of the complex self-adjoint operator $\mathcal{E}^{\prime \prime}(R)$, where
\begin{equation}\label{Est}
\mathcal{E}(u)=\frac{1}{2}\int_{\mathbb{R}^{N}}\left|\nabla u\right|^{2}dx-\int_{\mathbb{R}^{N}}\left|u\right|^{2}\mbox{Log}\left|u\right|^{2}dx.
\end{equation}
Notice that $\mathcal{E}$ is of class $C^1$ on $\Sigma({\mathbb R}^{N})$. Since  $\mathcal{E}^{\prime \prime}(R)$ is a bounded operator defined on $\Sigma(\mathbb{R}^{N})$ with values in $\Sigma^{\prime}(\mathbb{R}^{N})$ (see Section \ref{S:2} for more details), it is natural to assume that the energy functional is of class $C^{2}$ in a neighbourhood $V_{\epsilon}(R)$ of $R$, of size $\epsilon>0$, where 
\begin{equation*}
V_{\epsilon}(R):=\left\{u\in \Sigma(\mathbb{R}^{N}): \|u-R\|^{2}_{H^{1}}<\epsilon\right\}.
\end{equation*}

\begin{remark}\rm
The proof that the functional $\mathcal{E}$  is smooth on $V_{\epsilon}(R)$ seems very difficult because of the technical 
complications related to the singularity of the logarithm at the origin.
\end{remark}

\begin{open}\rm
Prove or disprove that 	$\mathcal{E}$ is of class $C^2$ on $V_{\epsilon}(R)$.
\end{open}

\begin{proposition} \label{DP2}
Suppose that $\mathcal{E}$ is of class $C^{2}$ on $V_{\epsilon}(R)$, for any $\epsilon$ small enough. Then the modulational stability property holds. That is, there exist two constants $C>0$ and $h>0$, such that 
\begin{equation*}
\inf_{y\in {\mathbb R}^{N}, \theta\in {\mathbb R}}\|\phi-e^{i\,\theta}R(\cdot-y)\|^{2}_{H^{1}}\leq C(\mathcal{E}(\phi)-\mathcal{E}(R))
\end{equation*}
for all $\phi\in \Sigma(\mathbb{R}^{N})$, such that $\|\phi\|^{2}_{L^{2}}=\|R\|^{2}_{L^{2}}$ and $\mathcal{E}(\phi)-\mathcal{E}(R)<h$.
\end{proposition}

In light of Proposition \ref{DP2}, the soliton dynamics in Theorem \ref{T1} can be improved. Indeed, we have the following result.

\begin{theorem} \label{T2}
Let $u_{{\epsilon}}\in \Sigma(\mathbb{R}^{N})$ be the family  of solutions to the Cauchy problem \eqref{0NL} with initial data \eqref{ID}. 
Furthermore, let $(x(t), \nu(t))$ be the solution of the Hamiltonian system \eqref{HS}. Under the hypothesis of Proposition \ref{DP2}, there exist  
$\theta_{\epsilon}:\mathbb{R}^{+}\rightarrow [0, 2\pi]$ such that,
\begin{equation*}
u_{{\epsilon}}(x,t)=e^{\frac{i}{\epsilon}(\nu(t)\cdot x+\theta_{\epsilon}(t))}e^{\frac{1+N}{2}}e^{-\frac{1}{\epsilon^{2}}|{x-x(t)}|^{2}}+\omega_{\epsilon}(x,t),
\end{equation*}
locally uniformly in time $t\in \mathbb{R}$, where $\omega_{\epsilon}\in H_{\epsilon}^{1}$ and $\|\omega_{\epsilon}(t)\|^{}_{H_{\epsilon}^{1}}=\mathcal{O}(\epsilon)$, as $\epsilon\rightarrow 0$.
\end{theorem}

The  paper is organized as follows.  

In Section \ref{S:0} we provide,  by variational techniques,  a characterization of the Gausson $R$.

In Section \ref{S:1} we prove Theorem \ref{T1}. 

In Section \ref{S:2}, we show some delicate estimates for 
$\mathcal{E}^{\prime \prime}(R)$ (Proposition \ref{DP2}). 

Finally, in Section \ref{S:3} we give a sketch of proof of Theorem  \ref{T2}. 
\vskip6pt
\noindent
{\bf Notation.} $\left\langle \cdot , \cdot \right\rangle$ is the duality pairing between $B^{\prime}$ and $B$, where $B$ is a Banach space and $B^{\prime}$ is its dual. The space $L^{2}(\mathbb{R}^{N},\mathbb{C})$  will be denoted  by $L^{2}(\mathbb{R}^{N})$ and its norm by $\|\cdot\|_{L^{2}}$.  This space will be endowed  with the real scalar product
\begin{equation*}
 \left(u,v\right)_{L^{2}}=\mbox{Re}\int_{\mathbb{R}^{N}}u\overline{v}\, dx, \quad
 \text{for $u,v\in L^{2}\left(\mathbb{R}^{N}\right)$}.
\end{equation*}
We denote by $\|\cdot\|_{H^{1}}$ the $H_{\mathbb{C}}^{1}(\mathbb{R}^{N})$-norm. If $L$ is a linear operator acting on some space $\left\langle Lv,v\right\rangle$ denotes the value of the quadratic form associated with $L$ evaluated at $v$. Finally, throughout this paper, the letter $C$ will denote positive constants whose value may change form line to line.

\section{Variational analysis}
\label{S:0}

In this section we establish some results that will be used later in the paper. In particular, we provide a characterization of the Gausson $R$ as minimizer of the energy functional $\mathcal{E}$  among functions with the same mass. 

We first need to introduce some notation which facilitates the subsequent discussion. 
Following \cite{CL},  we define the functions  $\Phi$, $\Psi$ on $\left[0, \infty\right)$  by 
\begin{equation}\label{IFD}
\Phi(s):=
\begin{cases}
-s^{2}\,\mbox{Log}(s^{2}), &\text{if $0\leq s\leq e^{-3}$;}\\
3s^{2}+4e^{-3}s^{}-e^{-6}, &\text{if $ s\geq e^{-3}$;}
\end{cases}
\quad  \Psi(s):=F(s)+\Phi(s),
\end{equation}
where
\begin{equation*}
F(s):=s^{2}\mbox{Log}\, s^{2} \quad  \text{for all  $s\in\mathbb{R}$}.
\end{equation*}
Notice that $\Phi$ is a  Young function (see Lemma 1.3 in \cite{CL}). Then  we define the associate Orlicz space by considering
\begin{equation*}
L^{\Phi}(\mathbb{R}^{N}):=\left\{f\in L^{1}_{\rm loc}(\mathbb{R}^{N}) : \Phi(\left|f\right|)\in L^{1}_{}(\mathbb{R}^{N})\right\}, 
\end{equation*}
and  we endowed this space with Luxemburg norm
\begin{equation*}
\left\|f\right\|_{L^{\Phi}}={\inf}\left\{k>0: \int_{\mathbb{R}^{N}}\Phi\left(k^{-1}{\left|f(x)\right|}\right)dx\leq 1 \right\}.
\end{equation*}
In light of \cite[Lemma 2.1]{CL}, we have that  $\left(L^{\Phi}(\mathbb{R}^{N}),\|\cdot\|_{L^{\Phi}} \right)$ is a separable reflexive Banach space. Finally, we define the Banach space $(W(\mathbb{R}^{N}),\|\cdot\|^{}_{W({\mathbb{R}^{N}})})=
(H^{1}_{\mathbb{C}}(\mathbb{R}^{N})\cap L^{\Phi}(\mathbb{R}^{N}),\|\cdot\|^{}_{W({\mathbb{R}^{N}})})$ where 
$$
\|f\|^{}_{W({\mathbb{R}^{N}})}:=\|f\|^{}_{H^{1}({\mathbb{R}^{N}})}+\|f\|^{}_{L^{\Phi}}.
\,\,\quad \text{for all $f\in W({\mathbb R}^N)$.}
$$ 
Since $\mathcal{E}\in {C}^{1}(W({\mathbb{R}^{N}}), \mathbb{R})$ (see Proposition 2.7 in \cite{CL}), it follows from Lemma \ref{POii} (ii) below that  $\mathcal{E}\in {C}^{1}(\Sigma({\mathbb{R}^{N}}), \mathbb{R})$.

\begin{lemma} \label{POii}
The following assertions hold.\\
{\rm (i)} The embedding $\Sigma(\mathbb{R}^{N})\hookrightarrow L^{q}(\mathbb{R}^{N})$ is compact, where $2\leq q< 2^{\ast}$.\\
{\rm (ii)} The inclusion map $\Sigma(\mathbb{R}^{N})\hookrightarrow L^{2-\delta}(\mathbb{R}^{N})$  is continuous, where $\delta={1}/{N}$. In particular, the embedding $\Sigma(\mathbb{R}^{N})\hookrightarrow W(\mathbb{R}^{N})$  is continuous.
\end{lemma} 
\begin{proof} 
Statement (i) is proved in \cite[Lemma 3.1]{JZ2000}. Next let $u\in \Sigma(\mathbb{R}^{N})$. By  H\"{o}lder's inequality with conjugate exponents $2N/(2N-1)$, $2N$ we obtain 
\begin{equation*}
\int_{\mathbb{R}^{N}}|u(x)|^{2-\frac{1}{N}}dx\leq \left(\int_{\mathbb{R}^{N}}\frac{1}{(1+|x|^{2})^{\alpha}}dx\right)^{\frac{1}{2N}}\left(\int_{\mathbb{R}^{N}}(1+|x|^{2})|u(x)|^{2}dx\right)^{\frac{2N-1}{2N}},
\end{equation*}
where $\alpha=2N-1$. Since $\alpha>N/2$, it follows that there exists a constant $C>0$ depending only on $N$ such that $\|u\|_{L^{2-{1}/{N}}(\mathbb{R}^{N})}\leq C\|u\|_{\Sigma(\mathbb{R}^{N})}$; that is, the embedding 
\begin{equation}\label{qw}
\Sigma(\mathbb{R}^{N})\hookrightarrow L^{2-\delta}(\mathbb{R}^{N})
\end{equation}
is continuous. On the other hand, by \cite[Proposition 2.2]{CL}, we see that 
\begin{equation}\label{DB}
\int_{\mathbb{R}^{N}}\left|\Psi(\left|u\right|)- \Psi(\left|v\right|)\right|dx\leq C\left(1+ \left\|u\right\|^{2}_{H^{1}(\mathbb{R}^{N})}+ \left\|v\right\|^{2}_{H^{1}(\mathbb{R}^{N})} \right)\left\|u-v\right\|_{{L^{2}}}
\end{equation}
for all $u$, $v\in H^{1}(\mathbb{R}^{N})$. Moreover,  it follows from \eqref{IFD}  that  for every  $N\in \mathbb{N}$, there exists  $C>0$ depending only on $N$ such that
\begin{equation*}
\Phi(|z|)\leq C(|z|^{2+\frac{1}{N}}+|z|^{2-\frac{1}{N}})+\Psi(|z|) \quad \text{for any $z\in \mathbb{C}$}.
\end{equation*}
Since $2 <2+{1}/{N}< 2^{\ast}$, it follows from  statement (i), \eqref{qw}, and \eqref{DB}  that if $u_{n}\rightarrow 0$  strongly in $\Sigma(\mathbb{R}^{N})$, then $\int_{\mathbb{R}^{N}}\Phi(|u_{n}|)dx\rightarrow 0$ as $n$ goes to $+\infty$.  By  \cite[inequality (2.2) in Lemma 2.1]{CL}, we see that  $u_{n}\rightarrow 0$  strongly  in $L^{\Phi}(\mathbb{R}^{N})$. Therefore, the embedding $\Sigma(\mathbb{R}^{N})\hookrightarrow L^{\Phi}(\mathbb{R}^{N})$ is continuous. This concludes the proof.
\end{proof}

\vskip4pt
\noindent
We have the following.

\begin{proposition} \label{P1}
For every $\eta>0$, there exist $h>0$ such that, if $\phi\in \Sigma(\mathbb{R}^{N})$, $\|\phi\|^{2}_{L^{2}}=\|R\|^{2}_{L^{2}}$ and 
$\mathcal{E}(\phi)-\mathcal{E}(R)<h$, then
\begin{equation*}
\inf_{y\in {\mathbb R}^{N},\,\theta\in {\mathbb R}}\|\phi-e^{i\,\theta}R(\cdot-y)\|^{2}_{H^{1}}<\eta^{2}.
\end{equation*}
\end{proposition}

\noindent
Our next goal is to prove Proposition \ref{P1}. In this aim, we study the constrained problem
\begin{equation}\label{MP}
\ell_{R}:=\inf\left\{\mathcal{E}(u): u\in {W}(\mathbb{R}^{N}),\,\, \|u\|^{2}_{L^{2}}=\|R\|^{2}_{L^{2}}\right\},
\end{equation}
where $R$ is the Gausson defined in \eqref{ID}.

\begin{lemma} \label{L10}
Every minimizing sequence of $\ell_{R}$ is relativity compact in ${W}(\mathbb{R}^{N})$. More precisely, if a sequence  $\left\{ u_{n}\right\}\subset {W}(\mathbb{R}^{N})$ is such that $\|u_{n}\|^{2}_{L^{2}}=\|R\|^{2}_{L^{2}}$ and $\mathcal{E}(u_{n})\rightarrow \ell_{R}$ as 
$n\to\infty$, then there exist a family $\left\{y_{n}\right\}_{n\in\mathbb N}\subset \mathbb{R}^{N}$ and $\left\{\theta_{n}\right\}_{n\in\mathbb N}\subset \mathbb{R}$ such that, up to a subsequence, 
\begin{equation*}
e^{-i\theta_{n}}u_{n}(\cdot+y_{n})\rightarrow R\quad \text{strongly in ${W}(\mathbb{R}^{N})$.}
\end{equation*}
In particular, $\mathcal{E}(R)=\ell_{R}$ and $\|u_{n}-e^{i\theta_{n}}R(\cdot-y_{n})\|^{2}_{H^{1}}\rightarrow 0$, as $n$ goes to $+\infty$.
\end{lemma}

\noindent
Before giving the proof of Lemma \ref{L10}, we need to establish some preliminaries. We define the following functionals of class $C^{1}$ on $W(\mathbb{R}^{N})$:
\begin{align}\label{SF}
 S(u)&:=\frac{1}{4}\int_{\mathbb{R}^{N}}\left|\nabla u\right|^{2}dx+\int_{\mathbb{R}^{N}}\left|u\right|^{2}dx-\frac{1}{2}\int_{\mathbb{R}^{N}}\left|u\right|^{2}\mbox{Log}\left|u\right|^{2}dx,\\
  I(u)&:=\frac{1}{2}\int_{\mathbb{R}^{N}}\left|\nabla u\right|^{2}dx+\int_{\mathbb{R}^{N}}\left|u\right|^{2}dx-\int_{\mathbb{R}^{N}}\left|u\right|^{2}\mbox{Log}\left|u\right|^{2}dx. \nonumber
\end{align}
Notice that \eqref{EP} with $\omega=1$ is equivalent to $S^{\prime}(u)=0$, and $I(u)=\left\langle S^{\prime}(u),u\right\rangle$ is the so-called Nehari functional.
Finally, let us consider the minimization problem
\begin{align}
\begin{split}\label{MPE}
d&:={\inf}\left\{S(u):\, u\in W(\mathbb{R}^{N}) \setminus  \left\{0 \right\},  I(u)=0\right\} \\ 
&=\frac{1}{2}\,{\inf}\left\{\left\|u\right\|_{L^{2}}^{2}:u\in  W(\mathbb{R}^{N}) \setminus \left\{0 \right\},  I(u)= 0 \right\}.
\end{split}
\end{align}

Now we recall the following fact from \cite[Proposition 1.4 and Lemma 3.1]{AHA1}.
\begin{theorem} \label{AT}
Let $\left\{ u_{n}\right\}_{n\in\mathbb N}\subset W^{}(\mathbb{R}^{N})$ be a minimizing sequence for $d$. Then there exist a family $(y_{n})\subset \mathbb{R}^{N}$ and  $\left\{\theta_{n}\right\}_{n\in\mathbb N}\subset \mathbb{R}$ such that, up to a subsequence, $e^{-i\theta_{n}}u_{n}(\cdot+y_{n})\rightarrow R$ strongly in $W(\mathbb{R}^{N})$. In particular, $I(R)=0$ and $S(R)=d$.
\end{theorem}

\begin{lemma} \label{L12}
The following fact hold,
\begin{equation}\label{Mnm}
S({R})=\inf\left\{S(u): u\in {W}(\mathbb{R}^{N}), \quad\|u\|^{2}_{L^{2}}=\|R\|^{2}_{L^{2}}\right\}.
\end{equation}
In particular, $\mathcal{E}(R)=\ell_{R}$ where $\ell_{R}$ is defined in \eqref{MP}.
\end{lemma}
\begin{proof}
First, it is clear that $\inf\left\{S(u): u\in {W}(\mathbb{R}^{N}),\,\,\|u\|^{2}_{L^{2}}=\|R\|^{2}_{L^{2}}\right\}\leq S({R})$.
Next we claim that if $\|u\|^{2}_{L^{2}}=\|R\|^{2}_{L^{2}}$, then $I(u)\geq 0$, where  $I$ is the Nehari functional. Indeed, 
assume by contradiction that $I(u)<0$.  It is not hard to show that there exists $\lambda$, $0<\lambda<1$, such that $I(\lambda u)=0$. But then, $\|\lambda u\|^{2}_{L^{2}}<\|u\|^{2}_{L^{2}}=\|R\|^{2}_{L^{2}}$ and $I(\lambda u)=0$, which is a contradiction with Theorem \ref{AT}. Therefore, if $\|u\|^{2}_{L^{2}}=\|R\|^{2}_{L^{2}}$, then $I(u)\geq 0$  and
\begin{equation*}
S(u)=\frac{1}{2}I(u)+\frac{1}{2}\|u\|^{2}_{L^{2}}\geq \frac{1}{2}\|u\|^{2}_{L^{2}}=\frac{1}{2}\|R\|^{2}_{L^{2}}=S(R);
\end{equation*}
this implies \eqref{Mnm}. Finally, the proof of the last assertion of lemma  immediately follows from \eqref{Mnm}. This concludes the proof of Lemma \ref{L12}.
\end{proof}

\noindent
Now we give the proof of Lemma \ref{L10}.

\begin{proof}[\bf{Proof of Lemma \ref{L10}}]
Let $\left\{u_{n}\right\}_{n\in\mathbb N}\subset W(\mathbb{R}^{N})$ be a minimizing sequence for the value $\ell_{R}$, that is $\|u_{n}\|^{2}_{L^{2}}=\|R\|^{2}_{L^{2}}$ and $\mathcal{E}(u_{n})\to \ell_{R},$ as $n\to\infty$.  From Lemma~\ref{L12}, we have as $n\to\infty$ that 
\begin{equation*}
I(u_{n})=\mathcal{E}(u_{n})+\|u_{n}\|^{2}_{L^{2}}\rightarrow \mathcal{E}(R)+\|R\|^{2}_{L^{2}}=I(R)=0.
\end{equation*}
Then, by elementary computations, we can see that  there exists a sequence $\left\{\lambda_{n}\right\}_{n\in\mathbb N}\subset \mathbb{R}^{+}$ such that $I(\lambda_{n} u_{n})=0$ and  $\lambda_{n}\rightarrow 1$.  Next, define the sequence $f_{n}(x)=\lambda_{n}u_{n}(x)$.  It is clear that $\|f_{n}\|^{2}_{L^{2}}\rightarrow\|R\|^{2}_{L^{2}}$ and $I(f_{n})=0$ for any $n\in\mathbb{N}$. Therefore, $\left\{f_{n}\right\}_{n\in\mathbb N}$  is a minimizing sequence for $d$. Thus, by Theorem \ref{AT}, up to a subsequence, there exist $\left\{y_{n}\right\}_{n\in\mathbb N}\subset \mathbb{R}^{N}$ and $\left\{\theta_{n}\right\}_{n\in\mathbb N}\subset \mathbb{R}$ such that $e^{-i\theta_{n}}f_{n}(\cdot+y_{n})\rightarrow R$  {strongly in ${W}(\mathbb{R}^{N})$.} Since  $\|f_{n}-u_{n}\|^{}_{{W}(\mathbb{R}^{N})}\rightarrow 0$ as $n\to\infty$, the assertion follows.
\end{proof}

\noindent
Now we give the proof of Proposition \ref{P1}.
\noindent 
\begin{proof}[\bf{Proof of Proposition \ref{P1}}]
The result is proved  by contradiction. Assume that there exist  $\eta>0$ and a sequence $\left\{\phi_{n}\right\}_{n\in\mathbb N}\subset \Sigma(\mathbb{R}^{N})$, such that $\|\phi_{n}\|^{2}_{L^{2}}=\|R\|^{2}_{L^{2}}$ and 
\begin{align}
& \mathcal{E}(\phi_{n})-\mathcal{E}(R)<\frac{1}{n}, \label{T21} \\
&\inf_{\theta\in \mathbb{R}}\inf_{y\in \mathbb{R}^{N}}  \big\|\phi_{n}-e^{i\theta}R(\cdot- y) \big\|_{H^{1}}\geq \eta , \quad \textrm{for any $n\in \mathbb{N}$.}\label{T22}
\end{align}
Since $ \mathcal{E}(\phi_{n})\geq \mathcal{E}(R)$, from formula \eqref{T21}, it follows that $\mathcal{E}(\phi_{n})\rightarrow \mathcal{E}(R)$ as $n\to \infty$. That is, 
$\left\{u_{n}\right\}_{n\in\mathbb N}$  is a minimizing sequence for $\ell_{R}$. By Lemma \ref{L10}, up to a subsequence, there exist $\left\{y_{n}\right\}_{n\in\mathbb N}\subset \mathbb{R}^{N}$ and $\left\{\theta_{n}\right\}_{n\in\mathbb N}\subset \mathbb{R}$ such that 
\begin{equation*}
\big\|\phi_{n}-e^{i\theta_{n}}R(\cdot- y_{n}) \big\|_{H^{1}}\rightarrow 0 \quad \text{as $n\to \infty$,}
\end{equation*}
which is a contradiction with \eqref{T22}. This ends the proof.
\end{proof}


\section{Dynamics of the Gausson}\label{S:1}
The main aim of this section is to prove Theorem \ref{T1}. Let  $u_{\epsilon}$ be a solution of the Cauchy problem \eqref{0NL} with initial data \eqref{ID}.
We define the momentum as a function $p_{\epsilon}:\mathbb{R}^{N}\times\mathbb{R}\rightarrow \mathbb{R}^{N}$ by setting 
\begin{equation*}
p_{\epsilon}(x,t):=\frac{1}{\epsilon^{N-1}}\text{Im}(\overline{u_{\epsilon}(x,t)}\nabla u_{\epsilon}(x,t)),
\end{equation*}
where $\text{Im}(z)$ denotes the imaginary part of $z$ and $\overline{u_{\epsilon}}$ is the complex conjugate of $u_{\epsilon}$.
Note that, by assumption, there exists $\omega>0$ such that $V(x)+\omega\geq 0$ for all $x\in \mathbb{R}^{N}$. In particular, we can assume $V\geq 0$. Indeed, if $u_{\epsilon}$ is a solution to \eqref{0NL}-\eqref{ID}, then  $e^{-i{\omega t}/{\epsilon}}u_{\epsilon}$ is a solution of \eqref{0NL}-\eqref{ID} with the potential $V(x)+\omega\geq 0$ instead of $V$.

\vskip3pt
\noindent
We have the following result.
\begin{lemma} \label{L1}
Let $u_{\epsilon}$ be the solution to problem \eqref{0NL} corresponding to the initial data \eqref{ID}. Then there exists a positive constant $C$, independent of $\epsilon>0$, such that 
$$
\sup_{t\in \mathbb{R}}\|\nabla u_{{\epsilon}}(t)\|^{2}_{L^{2}}\leq C\epsilon^{N-2}. 
$$
Moreover, there exists a positive constant $C_{1}$ such that
\begin{equation*}
\sup_{t\in \mathbb{R}}\Big|\int_{\mathbb{R}^{N}}p_{\epsilon}(x,t)dx\Big|\leq C_{1}.
\end{equation*}
 \end{lemma}

\begin{proof}
First, notice that  $E_{\epsilon}(u_{\epsilon}(x,t))=E_{\epsilon}(u_{\epsilon}(0,t))$. Moreover, by Lemma \ref{L4} below, there exists a constant $C>0$, independent of $\epsilon>0$, such that $E_{\epsilon}(u_{\epsilon}(x,t))\leq C$. Set $v_{\epsilon}(x):=u_{\epsilon}(\epsilon x)$. Since $V$ is  is nonnegative, we see that 
\begin{equation}\label{Dp1}
\frac{1}{2}\|\nabla v_{\epsilon}\|^{2}_{L^{2}}-\int_{\mathbb{R}^{N}}\left|v_{\epsilon}\right|^{2}\mbox{Log}\left|v_{\epsilon}\right|^{2}dx \leq C.
\end{equation}
Now, applying logarithmic Sobolev inequality (see \cite[Theorem 8.14]{ELL}) we have
\begin{equation*}
\int_{\mathbb{R}^{N}}\left|v_{\epsilon}\right|^{2}\mathrm{Log}\left|v_{\epsilon}\right|^{2}dx\leq \frac{\alpha^{2}}{\pi} \|\nabla v_{\epsilon} \|^{2}_{L^{2}}+\left(\mathrm{Log} \|v_{\epsilon}\|^{2}_{L^{2}}-N\left(1+\mathrm{Log}\,\alpha\right)\right)\|v_{\epsilon} \|^{2}_{L^{2}},
\end{equation*}
for any $\alpha>0$. By conservation the mass, we obtain that $\|v_{\epsilon}\|^{2}_{L^{2}}=m:=\|R\|^{2}_{L^{2}}$. Therefore,
\begin{equation*}
\left(\frac{1}{2}-\frac{\alpha^2}{\pi}\right)\|\nabla v_{\epsilon}\|^{2}_{L^{2}}\leq C+\left(\mathrm{Log}\,m-N\left(1+\mathrm{Log}\,\alpha\right)\right)m.
\end{equation*}
Taking $\alpha > 0$ sufficiently small, the first the assertion of lemma follows by rescaling.
On the other hand, by H\"older inequality, the mass conservation law and the first the assertion of lemma we see that
\begin{align*}
\left|\int_{\mathbb{R}^{N}}p_{\epsilon}(x,t)dx\right|\leq\int_{\mathbb{R}^{N}}|p_{\epsilon}(x,t)|dx\leq \frac{1}{\epsilon^{N/2}}\|u_{\epsilon}(\cdot, t)\|^{}_{L^{2}}\frac{1}{\epsilon^{N/2-1}}\|\nabla u_{\epsilon}(\cdot, t)\|^{}_{L^{2}}\leq C
\end{align*}
for every $t\in \mathbb{R}$, which completes the proof.
\end{proof}

\noindent
The following lemma will be useful later. For a proof see \cite[Lemma 3.3]{SK2}.
\begin{lemma} \label{L3}
Let $f\in C^{2}(\mathbb{R}^{N})$ be such that $\|f\|_{C^{2}(\mathbb{R}^{N})}<\infty$. Then for every $y\in \mathbb{R}^{N}$ fixed
\begin{equation*}
\int_{\mathbb{R}^{N}}f(\epsilon x+y) R^{2}(x)dx=f(y)\int_{\mathbb{R}^{N}}R^{2}(x)dx+\mathcal{O}(\epsilon^{2}), \quad \text{as $\epsilon\searrow 0$}, 
\end{equation*}
where $\mathcal{O}(\epsilon^{2})$ is independent of $y$.
\end{lemma}

\begin{lemma} \label{L4}
For every $t\in \mathbb{R}^{+}$ we have 
\begin{equation*}
E_{\epsilon}(u_{\epsilon}(t))=\mathcal{E}(R)+m\mathcal{H}(t)+\mathcal{O}(\epsilon^{2})\quad \text{as $\epsilon\searrow 0$.}
\end{equation*}
\end{lemma}
\begin{proof}
Since $R$ is real, it follows by the energy conservation law
\begin{equation*}
E_{\epsilon}(u_{\epsilon}(t))=E_{\epsilon}(u_{\epsilon}(0))=m\frac{|v_{0}|^{2}}{2}+\mathcal{E}(R)
+\int_{\mathbb{R}^{N}}V(\epsilon x+x_{0})R^{2}(x)dx.
\end{equation*}
Next, from Lemma \ref{L3} we see that
\begin{equation*}
\int_{\mathbb{R}^{N}}V(\epsilon x+x_{0}) R^{2}(x)dx=mV(x(0))+\mathcal{O}(\epsilon^{2}), \quad \text{as $\epsilon\searrow 0$.}
\end{equation*}
But then, by the conservation law of the function $t\mapsto \mathcal{H}(t)$, we obtain
\begin{align*}
E_{\epsilon}(u_{\epsilon}(t))&=\mathcal{E}(R)+m\mathcal{H}(0)+\mathcal{O}(\epsilon^{2})\\
&=\mathcal{E}(R)+m\mathcal{H}(t)+\mathcal{O}(\epsilon^{2}), \quad \text{as 
	$\epsilon\searrow 0$.}
\end{align*}
This completes the proof.
\end{proof}

\noindent
In our analysis, we use the following property of the functional $\delta_{x}$ defined on the space $C^{2}(\mathbb{R}^{N})$ endowed with the standard $C^2$ norm: there exist three constants $K_{0},K_1>0$ and 
$K_{2}>1$ such that if $\|\delta_{y}-\delta_{z}\|_{C^{2\ast}}\leq K_{0}$ then 
\begin{equation}\label{Dp}
K_{1}|y-z|\leq \|\delta_{y}-\delta_{z}\|_{C^{2\ast}}\leq K_{2}|y-z|.
\end{equation}
Here, $C^{2\ast}$  is the dual space of $C^{2}(\mathbb{R}^{N})$. For a proof of such statement, see \cite[Lemma 3.2]{SK2}.

 \vskip4pt
 \noindent
Let $\rho$ be a positive constant defined by 
\begin{equation}\label{cp}
\rho:=K_{2}\sup_{t\in [0, T]} |x(t)|+K_{0},
\end{equation}
where $T>0$,  $x(t)$ is defined in \eqref{HS}, $K_{0}$ and $K_{2}$ are as in \eqref{Dp}.
Observe that, as $K_2>1$, we have $|x(t)|\leq \rho$ for every $t\in [0,T]$.
Furthermore, let $\chi\in C^{\infty}(\mathbb{R}^{N})$ be function such that
\begin{equation}\label{cha}
\chi(x)=1 \quad \text{if $|x|<\rho$,} \quad \chi(x)=0 \quad \text{if $|x|>2\rho$.}
\end{equation}

\vskip2pt
\noindent
Then we have the following.
\begin{lemma} \label{L8}
Let $u_{{\epsilon}}$ be the family of solutions to problem \eqref{0NL} with initial data \eqref{ID} and consider the functions $\sigma_{{\epsilon}}:\mathbb{R} \rightarrow\mathbb{R}^{N}$ and $\lambda_{{\epsilon}}:\mathbb{R}\rightarrow\mathbb{R}$ defined by  
\begin{align*}
\sigma_{{\epsilon}}(t)&=\int_{\mathbb{R}^{N}}p_{\epsilon}(x,t)dx-m\nu(t),\\
\lambda_{{\epsilon}}(t)&=mV(x(t))-\frac{1}{\epsilon^{N}}\int_{\mathbb{R}^{N}}\chi(x)V(x)|u_{{\epsilon}}(x,t)|^{2}dx,
\end{align*}
where $\nu(t)$ is defined in \eqref{HS} and $m=\|R\|^{2}_{L^{2}}$. Then $\sigma_{{\epsilon}}(t)$ and $\lambda_{{\epsilon}}(t)$ are continuous on $\mathbb{R}$ and satisfy $\sigma_{{\epsilon}}(0)=0$, $|\lambda_{{\epsilon}}(0)|=\mathcal{O}(\epsilon^{2})$ as 
$\epsilon\searrow 0$.
\end{lemma}
\begin{proof}
The continuity of $\sigma_{{\epsilon}}$ and $\lambda_{{\epsilon}}$ follow from the regularity properties of the solution $u_{\epsilon}$. Since $R$ is a real function,  it follows easily that $\sigma_{{\epsilon}}(0)=0$. Finally, it is not hard to prove, using the Lemma \ref{L3},  that $|\lambda_{{\epsilon}}(0)|=\mathcal{O}(\epsilon^{2})$ as $\epsilon$ goes to zero. See e.g.\cite[Lemma 3.7]{EMBPMS} for more details.
\end{proof}

\noindent
Define now
\begin{equation}\label{af}
\psi_{{\epsilon}}(x,t):=e^{-\frac{i}{\epsilon}\nu(t)\cdot[\epsilon x+x(t)]}u_{{\epsilon}}(\epsilon x+x(t),t),
\end{equation}
where $(x(t), \nu(t))$ is the solution to problem \eqref{HS}. Notice that $\psi_{{\epsilon}}\in \Sigma(\mathbb{R}^{N})$ for every $t\in \mathbb{R}$ and $\epsilon>0$. Moreover, the mass of $\psi_{{\epsilon}}$ is conserved. Indeed, by a change of variable we see that
\begin{equation}\label{MC}
\|\psi_{{\epsilon}}(t)\|^{2}_{L^{2}}=\|u_{{\epsilon}}(\epsilon x+x(t),t)\|^{2}_{L^{2}}=\frac{1}{\epsilon^{N}}\|u_{{\epsilon}}(t)\|^{2}_{L^{2}}=\|R\|^{2}_{L^{2}}.
\end{equation}


\begin{lemma} \label{L9}
For every $t\in \mathbb{R}$ and $\epsilon>0$,
\begin{equation*}
0\leq \mathcal{E}(\psi_{{\epsilon}}(t))-\mathcal{E}(R)\leq |\nu(t)||\sigma_{{\epsilon}}(t)|+|\lambda_{{\epsilon}}(t)|+\mathcal{O}(\epsilon^{2}),
\end{equation*}
where $\psi_{{\epsilon}}$ is defined in \eqref{af}.
\end{lemma}
\begin{proof}
We recall that $p_{\epsilon}(x,t)=\text{Im}(\overline{u_{\epsilon}(x,t)}\nabla u_{\epsilon}(x,t))/\epsilon^{N-1}$. By a change of variable, it follows 
\begin{align*}
\mathcal{E}(\psi_{{\epsilon}}(t))=&\frac{1}{2\epsilon^{N-2}}\int_{\mathbb{R}^{N}}\left|\nabla u_{\epsilon}\right|^{2}dx+\frac{1}{2}m|\nu(t)|^{2}
-\frac{1}{\epsilon^{N}}\int_{\mathbb{R}^{N}}\left|u_{\epsilon}\right|^{2}\mbox{Log}\left|u_{\epsilon}\right|^{2}dx,\\
&-\nu(t)\cdot\int_{\mathbb{R}^{N}}p_{\epsilon}(x,t)dx.
\end{align*}
Thus, taking into account the definition of $E_{\epsilon}$, we have
\begin{align*}
\mathcal{E}(\psi_{{\epsilon}}(t))=E_{\epsilon}(u_{\epsilon}(t))-\frac{1}{\epsilon^{N}}\int_{\mathbb{R}^{N}}V(x)|u_{\epsilon}(t,x)|^{2}dx+\frac{1}{2}m|\nu(t)|^{2}-\nu(t)\cdot\int_{\mathbb{R}^{N}}p_{\epsilon}(x,t)dx.
\end{align*}
Then, from Lemma \ref{L4} we get
\begin{align*}
\mathcal{E}(\psi_{{\epsilon}}(t))-\mathcal{E}(R)&=m\mathcal{H}(t)-\frac{1}{\epsilon^{N}}\int_{\mathbb{R}^{N}}V(x)|u_{\epsilon}(t,x)|^{2}dx\\
&+\frac{1}{2}m|\nu(t)|^{2}-\nu(t)\cdot\int_{\mathbb{R}^{N}}p_{\epsilon}(x,t)dx+\mathcal{O}(\epsilon^{2})\\
&\leq |\nu(t)||\sigma_{\epsilon}(t)|+|\lambda_{\epsilon}(t)|-\frac{1}{\epsilon^{N}}\int_{\mathbb{R}^{N}}(1-\chi(x))V(x)|u_{\epsilon}(t,x)|^{2}dx+\mathcal{O}(\epsilon^{2}).
\end{align*}
Since $V\geq 0$, it follows that 
\begin{equation*}
0\leq \mathcal{E}(\psi_{{\epsilon}}(t))-\mathcal{E}(R)\leq |\nu(t)||\sigma_{{\epsilon}}(t)|+|\lambda_{{\epsilon}}(t)|+\mathcal{O}(\epsilon^{2}),
\end{equation*}
which concludes the proof of lemma.
\end{proof}

\begin{proof}[\bf{Proof of Theorem \ref{T1}}] 
Set $\beta_{\epsilon}(t):=|\nu(t)||\sigma_{{\epsilon}}(t)|+|\lambda_{{\epsilon}}(t)|$ for $t\in \mathbb{R}$. From Lemma \ref{L8},  since $\sup_{t\in \mathbb{R}}|\nu(t)|<+\infty$, it follows that the function $\left\{t\rightarrow \beta_{\epsilon}(t)\right\}$ is continuous and $|\beta_{\epsilon}(0)|=\mathcal{O}(\epsilon^{2})$ as $\epsilon\searrow 0$. 
Let $\eta>0$. Let us fix a time $T_{0}>0$. 
Let $h>0$, depending on $\eta$, be as in Proposition \ref{P1}. Introduce the number 
\begin{equation}\label{Tim}
T^{\ast}_{\epsilon}:=\sup\left\{t\in [0, T_{0}]:\beta_{\epsilon}(\tau)\leq \frac{h}{2},
\,\,\, \text{for all $\tau\in [0,t]$}\right\}.
\end{equation}
Since $|\beta_{\epsilon}(0)|=\mathcal{O}(\epsilon^{2})$ it follows that $T^{\ast}_{\epsilon}>0,$ for any $\epsilon>0$  small. By choosing $\epsilon$ sufficiently small, from Lemma \ref{L9}, we get
for all $t\in [0,T^*_\epsilon)$
\begin{equation*}
\mathcal{E}(\psi_{{\epsilon}}(t))-\mathcal{E}(R)\leq \beta_{\epsilon}(t)+\mathcal{O}(\epsilon^{2})<h.
\end{equation*}
Since $\psi_{{\epsilon}}\in \Sigma(\mathbb{R}^{N})$ and $\|\psi_{{\epsilon}}(t)\|^{2}_{L^{2}}=\|R\|^{2}_{L^{2}}$,  by Proposition \ref{P1} 
there exist two families of uniformly bounded functions ${\theta^{\ast}_{\epsilon}}:\mathbb{R}\rightarrow \mathbb{R}$ and ${z_{\epsilon}}:\mathbb{R}\rightarrow \mathbb{R}^{N}$ such that
\begin{equation*}
\|e^{-\frac{i}{\epsilon}\nu(t)\cdot[\epsilon x+x(t)]}u_{\epsilon}(\epsilon x+x(t),t)-e^{i{\theta^{\ast}_{\epsilon}}(t)}R(x+{z_{\epsilon}}(t))\|^{2}_{H^{1}}<\eta^{2}
\end{equation*}
for every $t\in [0,T^{\ast}_{\epsilon})$. Finally, by rescaling and setting 
$$
\theta_{\epsilon}(t):=\epsilon{\theta^{\ast}_{\epsilon}}(t), \quad 
y_{\epsilon}(t):=x(t)-\epsilon z_{\epsilon}(t),
$$ 
we get formula~\eqref{Ap}. The proof of Theorem \ref{T1} is complete.
\end{proof}


\section{Spectral analysis of linearization}\label{S:2}
The goal of this section is to prove Proposition \ref{DP2}. Before giving the proof, we need to establish some preliminary lemmas.

\begin{lemma} \label{L15}
Let $R$ be the Gausson \eqref{ID}. Then  there exist a positive constant $\delta$ such that for every $w\in  \Sigma(\mathbb{R}^{N})$ satisfying 
\begin{equation}\label{CHO}
\left(w, R\right)_{L^{2}}=\left(w, iR\right)_{L^{2}}
=\left(w, \frac{\partial R}{\partial x_{j}}\right)_{L^{2}}=0,\quad \text{for all $j=1$, \ldots $N$}
\end{equation}
we have 
\begin{equation*}
\left\langle S^{\prime\prime}({R})w, w\right\rangle\geq \delta \|\omega\|^{2}_{\Sigma(\mathbb{R}^{N})},
\end{equation*}
where the functional $S$ is defined in \eqref{SF}.
\end{lemma}

We set $w=u+iv$ for real valued functions $u$, $v\in \Sigma(\mathbb{R}^{N}, \mathbb{R})$. Then it is not hard to show that $S^{\prime\prime}({R})$ can be separated into a real and imaginary part $L_{+}$ and $L_{-}$ such that
\begin{equation*}
\left\langle S^{\prime\prime}({R})w, w\right\rangle=\left\langle L_{+}u,u\right\rangle+\left\langle L_{-}v,v\right\rangle,
\end{equation*}
where $L_{+}$ and $L_{-}$ are two bounded operator on $\Sigma(\mathbb{R}^{N})$ with values in $\Sigma^{\prime}(\mathbb{R}^{N})$ and given by
\begin{align*}
L_{+}u&=-\frac{1}{2}\Delta u+2|x|^{2}u-(N+2)u,\\
L_{-}v&=-\frac{1}{2}\Delta v+2|x|^{2}v-Nv.
\end{align*}
Indeed,  let $f\in \Sigma(\mathbb{R}^{N})$.  Then we see that
\begin{align*}
S^{\prime\prime}(f)w&=-\frac{1}{2}\Delta w+ w-\mathrm{Log}\left|f \right|^{2}w-2\frac{f}{|f|^{2}}\mbox{Re}(f \overline{w}).
\end{align*}
Now recalling the definition of $R$ given in \eqref{ID} we get 
\begin{align*}
S^{\prime\prime}(R)w&=-\frac{1}{2}\Delta w+ w-\mathrm{Log}\left|R \right|^{2}w-2\mbox{Re}(\overline{w})\\
&=-\frac{1}{2}\Delta w- Nw+2|x|^{2}w-2\mbox{Re}(\overline{w})\\
&=L_{+}u+iL_{-}v.
\end{align*}
The above lemma follows immediately from the two following lemmas.
\begin{lemma} \label{L16} 
There exists a positive constant $\delta_{1}$ such that for every $v\in  \Sigma(\mathbb{R}^{N})$ satisfying 
\begin{equation*}
\left(v, R\right)_{L^{2}}=0
\end{equation*}
we have $\left\langle L_{-}v, v\right\rangle\geq \delta_{1} \|v\|^{2}_{\Sigma(\mathbb{R}^{N})}$.\\
\end{lemma}

\begin{lemma} \label{L17} 
There exists a positive constant $\delta_{2}$ such that for every $u\in  \Sigma(\mathbb{R}^{N})$ satisfying 
\begin{equation*}
\left(u, R\right)_{L^{2}}=\left(u, \frac{\partial R}{\partial x_{j}}\right)_{L^{2}}=0\quad \text{for all $j=1$, \ldots $N$}
\end{equation*}
we have $\left\langle L_{+}u, u\right\rangle\geq \delta_{2} \|u\|^{2}_{\Sigma(\mathbb{R}^{N})}$.
\end{lemma}

Before giving the proof of Lemmas \ref{L16} and \ref{L17},  we discuss some spectral properties of $L_{-}$ and $L_{+}$. First, since $|x|^{2}\rightarrow +\infty$ as $|x|$ goes to $+\infty$, the operators $L_{-}$ and $L_{+}$ have only  discrete spectrum, i.e.\ $\sigma(L_{\pm})=\sigma_{p}(L_{\pm})=\left\{\lambda^{\pm}_{i}\right\}_{i\in \mathbb{N}}$. Moreover, the corresponding eigenvalues $\lambda^{\pm}_{i}$ tending to $+\infty$ and the eigenfunctions are exponentially decaying as $|x|\rightarrow +\infty$. 

\vskip4pt
\noindent
Now we give the proof of Lemma \ref{L16}.
\begin{proof}[\bf{Proof of Lemma \ref{L16}}] 
We claim that there exists $\kappa>0$  such that for every $v\in  \Sigma(\mathbb{R}^{N})\setminus\left\{0\right\}$ satisfying $\left(v, R\right)^{}_{L^{2}}=0$, we have 
\begin{equation}\label{AXIl}
\left\langle L_{-}v, v\right\rangle\geq \kappa \|v\|^{2}_{L^{2}}.
\end{equation}
Indeed, it is not hard to show that $L_{-}(R)=0$. Since $R>0$, it follows that $0$  is the first simple eigenvalue of $L_{-}$ (see \cite[Chapter 3]{EBS}). Next let $v\in \Sigma(\mathbb{R}^{N})\setminus\left\{0\right\}$  with $\left(v, R\right)^{}_{L^{2}}=0$. From the min-max characterization of eigenvalues (see \cite[Supplement 1]{EBS}) there exist $\kappa>0$ such that $\left\langle L_{-}v, v\right\rangle\geq \kappa \|v\|^{2}_{L^{2}}$. Notice that $\kappa$ is the second eigenvalue of $L_{-}$. This proves our claim. 

\vskip2pt
\noindent
Now, let us set
\begin{equation*}
\tau=\inf\left\{\left\langle L_{-}v, v\right\rangle: \|v\|^{2}_{\Sigma(\mathbb{R}^{N})}=1, \,\, \left(v, R\right)_{L^{2}}=0\right\},
\end{equation*}
and assume by contradiction that $\tau=0$. Let $\left\{v_{n}\right\}_{n\in\mathbb N}$ be a minimizing sequence for $\tau$. Then since  $\|v_{n}\|_{\Sigma(\mathbb{R}^{N})}=1$, we can assume that the sequence converges weakly in $\Sigma(\mathbb{R}^{N})$ to some $v$. Furthermore, since the embedding $\Sigma(\mathbb{R}^{N})\hookrightarrow L^{2}(\mathbb{R}^{N})$ is compact, it follows that $v_{n}\rightarrow v$ 
strongly in $L^{2}(\mathbb{R}^{N})$ as $n\to\infty$. In particular, $\left(v, R\right)^{}_{L^{2}}=0$ and 
\begin{equation*}
0\leq\left\langle L_{-}v, v\right\rangle\leq\liminf_{n\rightarrow\infty}\left\langle L_{-}v_{n}, v_{n}\right\rangle=0.
\end{equation*}
This implies that $\left\langle L_{-}v, v\right\rangle=0$ and, since $\left(v, R\right)^{}_{L^{2}}=0$, it follows from \eqref{AXIl} that $v\equiv 0$.  On the other hand,
\begin{align*}
\lim_{n\rightarrow\infty}\|v_{n}\|^{2}_{L^{2}}&=\lim_{n\rightarrow\infty}\left\{\|v_{n}\|^{2}_{L^{2}}+\left\langle L_{-}v_{n}, v_{n}\right\rangle \right\}\geq \lim_{n\rightarrow\infty}\left\{\frac{1}{2}\|v_{n}\|^{2}_{\Sigma(\mathbb{R}^{N})}-N\|v_{n}\|^{2}_{L^{2}}\right\}\\
&=\frac{1}{2}-N\lim_{n\rightarrow\infty}\|v_{n}\|^{2}_{L^{2}}.
\end{align*}
Therefore, $\|v\|^{2}_{L^{2}}\geq 1/(2N+2)$,  which is a contradiction to the fact that $v\equiv 0$. This completes of proof of lemma.
\end{proof}

We now turn our attention to $L_{+}$. In order to prove Lemma \ref{L17}, we first establish the following.
\begin{lemma} \label{L18} 
The following assertions hold.\\
{\rm (i)} The operator $L_{+}$ has only one negative eigenvalue. The unique negative simple
eigenvalue equals $-2$, and $R$ is the corresponding eigenfunction.\\
{\rm (ii)} The second eigenvalue of $L_{+}$ is $0$ and 
\begin{equation*}
\mathrm{ker}(L_{+})=\mathrm{span}\left\{\frac{\partial R}{\partial x_{j}}: j=1,\ldots N\right\}. 
\end{equation*}
\end{lemma}
\begin{proof}
We remark that $L_{+}(R)=-2R$. But then, since $R>0$, it follows that $-2$ is the first simple eigenvalue of $L_{+}$. From \cite[Example 4.5 and Section 4.2]{MAS}, we see that $-2$ is the only negative eigenvalue. On the other hand, it is clear that $0$ is the next eigenvalue. Indeed, since $R\in {C}^{\infty}(\mathbb{R}^{N})$, an easy calculation shows that $L_{+}(\partial R/\partial x_{j})=0$. Thus, $0$ is an eigenvalue of $L_{+}$ and 
\begin{equation*}
\mathrm{span}\left\{\frac{\partial R}{\partial x_{j}}: j=1,\ldots N\right\}\subseteq\mathrm{ker}(L_{+}). 
\end{equation*}
The reverse inclusion is shown in \cite[Theorem 1.3]{PMS}. The proof is completed.
\end{proof}

\begin{proof}[\bf{Proof of Lemma \ref{L17}}] 
From spectral decomposition theorem \cite[p. 177]{TKP} and Lemma \ref{L17}, the space $\Sigma(\mathbb{R}^{N})$ can be decomposed into 
$\Sigma(\mathbb{R}^{N})=\mathrm{span}\left\{R/\|R\|^{}_{L^{2}}\right\}\oplus \mathrm{ker}(L_{+})\oplus E_{+}$, where $E_{+}$ is the image of the spectral projection to the positive part of the spectrum of $L_{+}$. In particular, if $\xi\in E_{+}$, then $\left\langle L_{+}\xi,\xi\right\rangle>0$. Then for every $u\in \Sigma(\mathbb{R}^{N})\setminus\left\{0\right\}$ with 
\begin{equation*}
\left(u, R\right)_{L^{2}}=\left(u, \frac{\partial R}{\partial x_{j}}\right)_{L^{2}}=0,\quad \text{for all $j=1$, \ldots $N$}
\end{equation*}
we get $\left\langle L_{+}u,u\right\rangle>0$. The remainder of the argument is a literal repetition of the proof of Lemma \ref{L16}.  We omit the details.
\end{proof}

\begin{lemma} \label{L20} 
Let $\psi\in \Sigma(\mathbb{R}^{N})$ such that 
$$
\inf_{y\in {\mathbb R}^{N}, \theta\in [0, 2\pi)}\|\psi-e^{-i\,\theta}R(\cdot+y)\|_{L^{2}}\leq \|R\|_{L^{2}}
$$ 
and  $\|\psi\|^{2}_{L^{2}}=\|R\|^{2}_{L^{2}}$. Then 
\begin{equation*}
\inf_{y\in {\mathbb R}^{N}, \theta\in [0, 2\pi)}\|\psi-e^{-i\,\theta}R(\cdot+y)\|^{2}_{L^{2}}
\end{equation*}
is achieved at some $y_{0}\in {\mathbb R}^{N}$ and $\theta_{0}\in [0, 2\pi)$.
\end{lemma}
\begin{proof}
Consider the  auxiliary function $\Gamma:{\mathbb R}^{N}\times [0, 2\pi)\rightarrow \mathbb R$ defined by 
$$
\Gamma(y,\theta):=\|\psi-e^{-i\,\theta}R(\cdot+y)\|_{L^{2}}. 
$$
It is clear that $\Gamma$ is a continuous function. Now, since $\|\psi\|^{2}_{L^{2}}=\|R\|^{2}_{L^{2}}$ we see that
\begin{equation*}
\Gamma^2(y,\theta)=2\|R\|^{2}_{L^{2}}-2\mbox{Re} \left( e^{-i\,\theta}\int_{\mathbb{R}^{N}}R(x+y)\overline{\psi(x)}dx\right).
\end{equation*}
Since $R(\cdot+y)$ decay exponentially to zero as $|y|\rightarrow +\infty$, we have  $R(\cdot+y)\rightharpoonup 0$ in $L^{2}({\mathbb R}^{N})$ as $|y|$ goes to $+\infty$. Thus,
we have
\begin{equation*}
\lim_{|y|\rightarrow \infty}\Gamma^2(y,\theta)=2\|R\|^{2}_{L^{2}}.
\end{equation*}
By the first assumption on the function $\psi$, for every $\delta>0$, we see that there exist points  $y^{{\ast}}\in {\mathbb R}^{N}$ and $\theta^{{\ast}}\in [0, 2\pi)$ such that $\Gamma(y^{\ast},\theta^{\ast})\leq \|R\|_{L^{2}}+\delta$. As a consequence,  $\Gamma(y,\theta)$ attains its infimum over the compact set $\overline{B}_{\varrho}(0)\times[0, 2\pi]$, for a suitable $\varrho>0$, which finishes the proof.
\end{proof}


\vskip2pt
\noindent
We define now the a tubular neighbourhood of $R$ of size $\epsilon\in (0,\|R\|^{2}_{L^{2}})$ by
\begin{equation*}
U_{\epsilon}(R):=\left\{u\in\Sigma(\mathbb{R}^{N}):\|u\|^{2}_{L^{2}}=\|R\|^{2}_{L^{2}},\,\,  \inf_{y\in {\mathbb R}^{N}, \theta\in [0, 2\pi)}\|e^{i\,\theta}u(\cdot-y)-R\|^{2}_{H^{1}}<\epsilon\right\}.
\end{equation*}
By Lemma \ref{L20}, there exist $\sigma: U_{\epsilon}(R)\rightarrow \mathbb R$ and $Y: U_{\epsilon}(R)\rightarrow \mathbb{R}^{N}$ such that, for all $u\in U_{\epsilon}(R)$,
\begin{equation}\label{IG}
\|e^{i\,\sigma(u)}u(\cdot-Y(u))-R\|^{2}_{L^{2}}=\inf_{y\in {\mathbb R}^{N}, \theta\in [0, 2\pi)}\|e^{i\,\theta}u(\cdot-y)-R\|^{2}_{L^{2}}
\end{equation}
We claim that the function $w:=e^{i\,\sigma(u)}u(\cdot-Y(u))$ satisfies the orthogonality conditions
\begin{equation}\label{Fw}
\left(w, iR\right)_{L^{2}}=\left(w, \frac{\partial R}{\partial x_{j}}\right)_{L^{2}}=0 \quad \text{for all $j=1$, \ldots $N$.}
\end{equation}
Indeed, by differentiating \eqref{IG} with respect to $\theta$  we see that
\begin{align*}
\left(w, iR\right)_{L^{2}}&=-\mbox{Re}\int_{{\mathbb R}^{N}}\overline{ie^{i\sigma(u)}u(\cdot-Y(u))}R(x)\,dx\\
&=\left(ie^{i\sigma(u)}u(\cdot-Y(u)), e^{i\,\sigma(u)}u(\cdot-Y(u))-R\right)_{L^{2}}=0.
\end{align*}
On the other hand,  by differentiating \eqref{IG} with respect to $y_{j}$ we get
\begin{align*}
\left(w, \frac{\partial R}{\partial x_{j}}\right)_{L^{2}}&=\mbox{Re}\int_{{\mathbb R}^{N}}\overline{e^{i\sigma(u)}u(\cdot-Y(u))}\frac{\partial R(x)}{\partial x_{j}}\,dx-\int_{{\mathbb R}^{N}}R(x)\frac{\partial R(x)}{\partial x_{j}}\, dx\\
&=\left(u-e^{-i\,\sigma(u)}R(\cdot+Y(u)),  e^{-i\sigma(u)}\frac{\partial R}{\partial x_{j}}(\cdot+Y(u))\right)_{L^{2}}=0.
\end{align*}

Now we give the proof of Proposition \ref{DP2}.

\begin{proof}[\bf{Proof of Proposition \ref{DP2}}] 
Our proof is inspired by the one contained in \cite[Lemma 6.3]{SLC}. First, we claim that there exist $\epsilon>0$ and $C>0$ such that for all $v\in U_{\epsilon}(R)$ we have
\begin{equation*}
\mathcal{E}(v)-\mathcal{E}(R)\geq C\inf_{y\in {\mathbb R}^{N}, \theta\in {\mathbb R}}\|e^{i\,\theta}v(\cdot-y)-R\|^{2}_{H^{1}}.
\end{equation*}
Indeed, for $\epsilon$ small enough, let $w:=e^{i\,\sigma(v)}v(\cdot-Y(v))$ be as in \eqref{IG}. Let $\lambda\in {\mathbb R}$ and $z\in \Sigma(\mathbb{R}^{N})$ be such that $w=R+\lambda R+z$ with $(z, R)_{L^{2}}=0$. From \eqref{Fw} we see that 
$$
(z, \partial R/ \partial x_{j})_{L^{2}}=0,\quad (z, iR)_{L^{2}}=0. 
$$
Then $z$ satisfies the conditions \eqref{CHO} in Lemma \ref{L15}. Hence, 
there exists $\delta>0$ such that 
\begin{equation}\label{DSS}
\left\langle S^{\prime\prime}({R})z, z\right\rangle\geq \delta \|z\|^{2}_{\Sigma(\mathbb{R}^{N})}\geq \delta\|z\|^{2}_{H^{1}}.
\end{equation}
On the other hand, since $S^{\prime}(R)=0$, under the hypothesis of Proposition \ref{DP2} and by virtue of Taylor formula we get
\begin{align}\nonumber
S(v)-S(R)&=S(w)-S(R)\\ \label{Na}
&=\frac{1}{2}\left\langle S^{\prime\prime}(R)(w-R), w-R\right\rangle+o(\|w-R\|^{2}_{H^{1}}).
\end{align}
Moreover, it is not hard to show that $\lambda=o(\|w-R\|_{H^{1}})$ and (see \cite[Lemma 6.3]{SLC})
\begin{equation*}
\left\langle S^{\prime\prime}(R)(w-R), w-R\right\rangle=\left\langle S^{\prime\prime}(R)z,z\right\rangle+o(\|w-R\|^{2}_{H^{1}}).
\end{equation*}
Notice that $\|z\|^{2}_{H^{1}}\geq\|w-R\|^{2}_{H^{1}}+o(\|w-R\|^{2}_{H^{1}})$. Thus, since $\|v\|^{2}_{L^{2}}=\|R\|^{2}_{L^{2}}$, from \eqref{DSS} and \eqref{Na} we obtain 
\begin{equation*}
\mathcal{E}(v)-\mathcal{E}(R)\geq \frac{\delta}{2}\|w-R\|^{2}_{H^{1}}+o(\|w-R\|^{2}_{H^{1}}).
\end{equation*}
Then choosing $\epsilon$ small enough and recalling the definition of $w$ in \eqref{IG}, it follows that 
\begin{equation*}
\mathcal{E}(v)-\mathcal{E}(R)\geq \frac{\delta}{4}\inf_{y\in {\mathbb R}^{N}, \theta\in {\mathbb R}}\|e^{i\,\theta}v(\cdot-y)-R\|^{2}_{H^{1}},
\end{equation*}
for every $v\in U_{\epsilon}(R)$. This concludes the proof of claim. 

Finally, from Proposition \ref{P1} we see that for every $\epsilon>0$, there exists $h>0$ such that if $\mathcal{E}(v)-\mathcal{E}(R)<h$ and $\|v\|^{2}_{L^{2}}=\|R\|^{2}_{L^{2}}$, then $v\in U_{\epsilon}(R)$. Then, choosing $h$ small enough, Proposition \ref{DP2} follows.
\end{proof}


\section{Proof of Theorem \ref{T2}}\label{S:3}
In this section we will show the main steps of the proof of Theorem \ref{T2}. We essentially follow the argument of \cite{SK2}, which is based upon the original paper by Bronski and Jerrard \cite{BJ}. Using the variational structure of \eqref{0NL} and by the regularity of solutions (see the paragraph after Proposition \ref{PCS}), it is not difficult to show that the solution $u_{\epsilon}$ satisfies the identities
\begin{align}\label{Id1}
\frac{1}{\epsilon^{N}}\frac{\partial |u_{\epsilon}|^{2}}{\partial t}(x,t)&=-\mbox{div}_{x}p_{\epsilon}(x,t),\quad (x,t)\in {\mathbb R}^{N}\times {\mathbb R}^{},\\\label{Id2}
\int_{{\mathbb R}^{N}}\frac{\partial p_{\epsilon}}{\partial t}(x,t)dx&=-\int_{{\mathbb R}^{N}}\nabla V(x)\frac{|u_{\epsilon}(x,t)|^{2}}{\epsilon^{N}}, \quad x\in {\mathbb R}^{N}.
\end{align}
In light of Proposition \ref{DP2}, the result in Theorem \ref{T1} can be improved. More precisely, combining  Proposition \ref{DP2} and Lemma \ref{L9} and following the same argument as Theorem \ref{T1}  we have the following 

\begin{proposition}\label{PA12}
If $u_{\epsilon}\in \Sigma(\mathbb{R}^{N})$ is the family  of solutions to the Cauchy problem \eqref{0NL} with initial data \eqref{ID}, then  there exist $\epsilon_{0}>0$, a time $T^{\ast}_{\epsilon}>0$, families of bounded functions $\theta_{\epsilon}: \mathbb{R}\rightarrow [0, 2\pi)$, $y_{\epsilon}:\mathbb{R}\rightarrow \mathbb{R}^{N}$ such that 
\begin{equation*}
u^{\epsilon}(x,t)=e^{\frac{i}{\epsilon}(\nu(t)\cdot x+\theta_{\epsilon}(t))}e^{\frac{1+N}{2}}e^{-\frac{1}{\epsilon^{2}}|{x-y_{\epsilon}(t)}|^{2}}+\omega_{\epsilon}(x,t),
\end{equation*}
where 
\begin{equation*}
\|\omega_{\epsilon}(t)\|^{2}_{H_{\epsilon}^{1}}\leq  C(|\nu(t)||\sigma_{{\epsilon}}(t)|+|\lambda_{{\epsilon}}(t)|)+\mathcal{O}(\epsilon^{2})\quad
\text{for all $\epsilon\in (0, \epsilon_{0})$ and $t\in [0, T^{\ast}_{\epsilon})$.}
\end{equation*}
\end{proposition}

\noindent
Let $\epsilon_{0}>0$, $T^{\ast}_{\epsilon}>0$ and $y_{\epsilon}(t)$  be as in Proposition \ref{PA12}. Then we have the following

\begin{lemma} \label{L23} 
There exists a constant $C>0$ such that
\begin{align*}
&\left\|\frac{|u_{\epsilon}(x,t)|^{2}}{\epsilon^{N}}dx-m\delta_{y_{\epsilon}(t)}\right\|_{(C^{2})^{\ast}}+\left\|p_{\epsilon}(x,t)dx-m\nu(t)\delta_{y_{\epsilon}(t)}\right\|_{(C^{2})^{\ast}}   \\
&\leq C|\sigma_{{\epsilon}}(t)|+C|\lambda_{{\epsilon}}(t)|+\mathcal{O}(\epsilon^{2})
\end{align*}
for all $\epsilon\in (0, \epsilon_{0})$ and $t\in [0, T^{\ast}_{\epsilon})$.
\end{lemma}
\begin{proof}
First, notice that for any $v\in H^{1}({\mathbb R}^{N})$, we have $|\nabla|v||^{2}=|\nabla v|^{2}-\frac{|\mbox{Im}(\overline{v}\nabla v)|^{2}}{|v|^{2}}$. Furthermore, it is not hard to show that 
\begin{align}\nonumber
\int_{{\mathbb R}^{N}}\frac{|\mbox{Im}(\overline{\psi_{\epsilon}}\nabla \psi_{\epsilon})|^{2}}{|\psi_{\epsilon}|^{2}}\,dx&=\epsilon^{N}\int_{{\mathbb R}^{N}}\frac{|p_{\epsilon}(x,t)|^{2}}{|u_{\epsilon}(x,t)|^{2}}dx+m|\nu(t)|^{2}-2\nu(t)\cdot\int_{{\mathbb R}^{N}}p_{\epsilon}(x,t)dx\\ \nonumber
&=\int_{{\mathbb R}^{N}}\left|\epsilon^{^{N/2}}\frac{p_{\epsilon}(x,t)}{|u_{\epsilon}(x,t)|}-\frac{1}{m}\left(\int_{{\mathbb R}^{N}}p_{\epsilon}(x,t)dx\right)\frac{|u_{\epsilon}(x,t)|}{\epsilon^{N/2}}\right|^{2}\,dx\\
&+m\left|\nu(t)-\frac{\int_{{\mathbb R}^{N}}p_{\epsilon}(x,t)dx}{m}\right|^{2},\label{Np}
\end{align}
where $\psi_{\epsilon}(x,t)$ is the function defined in \eqref{af}. By Lemma \ref{L9} we see that
\begin{align*}
0\leq \mathcal{E}(|\psi_{{\epsilon}}|)-\mathcal{E}(R)+\frac{1}{2}\int_{{\mathbb R}^{N}}\frac{|\mbox{Im}(\overline{\psi_{\epsilon}}\nabla \psi_{\epsilon})|^{2}}{|\psi_{\epsilon}|^{2}}\,dx\leq |\nu(t)||\sigma_{{\epsilon}}(t)|+|\lambda_{{\epsilon}}(t)|+\mathcal{O}(\epsilon^{2}).
\end{align*}
for every $t\in[0, T^{\ast}_{\epsilon})$ and $\epsilon \in (0, \epsilon_{0})$.  Since $\mathcal{E}(|\psi_{{\epsilon}}|)-\mathcal{E}(R)\geq 0$, it follows from \eqref{Np},
\begin{align}\nonumber
&\int_{{\mathbb R}^{N}}\left|\epsilon^{^{N/2}}\frac{p_{\epsilon}(x,t)}{|u_{\epsilon}(x,t)|}-\frac{1}{m}\left(\int_{{\mathbb R}^{N}}p_{\epsilon}(x,t)dx\right)\frac{|u_{\epsilon}(x,t)|}{\epsilon^{N/2}}\right|^{2}\,dx\\ \label{zx}
&+m\left|\nu(t)-\frac{\int_{{\mathbb R}^{N}}p_{\epsilon}(x,t)dx}{m}\right|^{2}\leq C|\nu(t)||\sigma_{{\epsilon}}(t)|+C|\lambda_{{\epsilon}}(t)|+\mathcal{O}(\epsilon^{2}).
\end{align}
for every $t\in[0, T^{\ast}_{\epsilon})$ and $\epsilon \in (0, \epsilon_{0})$. Now, to prove the assertion, we need to estimate $\Gamma(t)$, where
\begin{align*}
\Gamma(t):&=\left|\int_{{\mathbb R}^{N}}f(x)\frac{|u_{\epsilon}(x,t)|^{2}}{\epsilon^{N}}dx-mf(y_{\epsilon})\right|+\left|\int_{{\mathbb R}^{N}}p_{\epsilon}(x,t)f(x)dx-m\nu(t)f(y_{\epsilon})\right|
\end{align*}
for every function $f$ in $C^{2}({\mathbb R}^{N})$ with $\|f\|_{C^{2}}\leq 1$. 
By simple computations we see that
\begin{align*}
&\left|\int_{{\mathbb R}^{N}}p_{\epsilon}(x,t)f(x)dx-m\nu(t)f(y_{\epsilon})\right|\leq \\
&\leq\frac{1}{m}\left|\int_{{\mathbb R}^{N}}p_{\epsilon}(x,t)dx\right|\left|\int_{{\mathbb R}^{N}}\frac{f(x)|u_{\epsilon}(x,t)|^{2}}{\epsilon^{N}}dx-mf(y_{\epsilon}(t))\right|\\
&+\left|\int_{{\mathbb R}^{N}}f(x)\left\{p_{\epsilon}(x,t)-\frac{1}{m}\left(\int_{{\mathbb R}^{N}}p_{\epsilon}(x,t)dx \right)\frac{|u_{\epsilon}(x,t)|^{2}}{\epsilon^{N}}
\right\}dx\right|+ C\sigma_{\epsilon}(t),
\end{align*}
for every $t\in[0, T^{\ast}_{\epsilon})$ and $\epsilon \in(0, \epsilon_{0})$. Here, the function $\sigma_{\epsilon}(t)$ is defined in Lemma \ref{L8}. Set $f_{\sharp}(x):=f(x)-f(y_{\epsilon}(t))$. Since $\int_{{\mathbb R}^{N}}p_{\epsilon}(x,t)dx$ is bounded (see Lemma \ref{L1}) and 
\begin{align*}
\int_{{\mathbb R}^{N}}\left\{p_{\epsilon}(x,t)-\frac{1}{m}\left(\int_{{\mathbb R}^{N}}p_{\epsilon}(x,t)dx \right)\frac{|u_{\epsilon}(x,t)|^{2}}{\epsilon^{N}}\right\}dx=0,
\end{align*}
it follows that
\begin{align*}
\Gamma(t)&\leq\int_{{\mathbb R}^{N}}f_{\sharp}(x)\left|p_{\epsilon}(x,t)-\frac{1}{m}\left(\int_{{\mathbb R}^{N}}p_{\epsilon}(x,t)dx \right)\frac{|u_{\epsilon}(x,t)|^{2}}{\epsilon^{N}}\right|dx\\
&+\int_{{\mathbb R}^{N}}|f_{\sharp}(x)|\frac{|u_{\epsilon}(x,t)|^{2}}{\epsilon^{N}}dx+C\int_{{\mathbb R}^{N}}|f_{\sharp}(x)|\frac{|u_{\epsilon}(x,t)|^{2}}{\epsilon^{N}}dx+C\sigma_{\epsilon}(t).
\end{align*}
Using the inequality $ab\leq a^{2}+b^{2}$ and \eqref{zx} we obtain
\begin{align*}
\Gamma(t)&\leq\frac{1}{2}\int_{{\mathbb R}^{N}}\left|\epsilon^{N/2}\frac{p_{\epsilon}(x,t)}{|u_{\epsilon}(x,t)|}-\frac{1}{m}\left(\int_{{\mathbb R}^{N}}p_{\epsilon}(x,t)dx \right)\frac{|u_{\epsilon}(x,t)|^{}}{\epsilon^{N/2}}\right|^{2}dx\\
&+\int_{{\mathbb R}^{N}}\left\{C|f_{\sharp}(x)|+\frac{1}{2}|f_{\sharp}(x)|^{2}\right\}\frac{|u_{\epsilon}(x,t)|^{2}}{\epsilon^{N}}dx+C\sigma_{\epsilon}(t)\\
&\leq \int_{{\mathbb R}^{N}}\left\{C|f_{\sharp}(x)|+\frac{1}{2}|f_{\sharp}(x)|^{2}\right\}\frac{|u_{\epsilon}(x,t)|^{2}}{\epsilon^{N}}dx+C|\sigma_{{\epsilon}}(t)|+C|\lambda_{{\epsilon}}(t)|+\mathcal{O}(\epsilon^{2})
\end{align*}
Finally, in view of the elementary inequality $a^{2}\leq 2b^{2}+2(a-b)^{2}$ with
\begin{align*}
a=\frac{|u_{\epsilon}(x,t)|}{\epsilon^{N/2}}, \quad b=\frac{1}{\epsilon^{N/2}}R\left(\frac{x-y_{\epsilon}(t)}{\epsilon}\right),
\end{align*}
since $f_{\sharp}(y_{\epsilon}(t))=0$, it follows from Lemma \ref{L3} and Proposition \ref{PA12},
\begin{align*}
\Gamma(t)&\leq\frac{C}{\epsilon^{N}}\int_{{\mathbb R}^{N}}\left\{|f_{\sharp}(x)|+|f_{\sharp}(x)|^{2}\right\}R^{2}\left(\frac{x-y_{\epsilon}(t)}{\epsilon}\right)\,dx\\ &+\frac{C}{\epsilon^{N}}\int_{{\mathbb R}^{N}}\left|u_{\epsilon}(x,t)-R\left(\frac{x-y_{\epsilon}(t)}{\epsilon}\right)\right|^{2}dx+C|\sigma_{{\epsilon}}(t)|+C|\lambda_{{\epsilon}}(t)|+\mathcal{O}(\epsilon^{2})\\
&\leq C|\sigma_{{\epsilon}}(t)|+C|\lambda_{{\epsilon}}(t)|+\mathcal{O}(\epsilon^{2}),
\end{align*}
for every $t\in[0, T^{\ast}_{\epsilon})$ and $\epsilon \in(0, \epsilon_{0})$, which concluded the proof.
\end{proof}

We now turn to the estimate the distance $|x(t)-y_{\epsilon}(t)|$, where the function $y_{\epsilon}(t)$ is given in Proposition \ref{PA12} and 
$x(t)$  is the solution of the classical Hamiltonian system \eqref{HS}.

\begin{lemma} \label{L25} 
Let $u_{{\epsilon}}$ be the family of solutions to problem \eqref{0NL} with initial data \eqref{ID}. Consider the function $\gamma_{{\epsilon}}$ defined by  
\begin{align*}
\gamma_{\epsilon}(t)=m\,x(t)-\frac{1}{\epsilon^{N}}\int_{{\mathbb R}^{N}}x\chi(x)|u_{\epsilon}(x,t)|^{2}dx,
\end{align*}
where $\chi(x)$ is defined in \eqref{cha}. Then $\gamma_{{\epsilon}}(t)$ is a continuous function on $\mathbb{R}$ and satisfy  $|\gamma_{{\epsilon}}(0)|=\mathcal{O}(\epsilon^{2})$ as $\epsilon$ goes to zero.
\end{lemma}
\begin{proof}
The proof easily follows from Lemma \ref{L3}, and the properties of the functions $u_{\epsilon}(x,t)$ and $\chi(x)$.
\end{proof}

\begin{lemma} \label{L26} 
Let $T_{\epsilon}^{\ast}>0$ be the time introduced in \eqref{Tim}. There exist positive constants $h_{0}$ and $\epsilon_{0}$, such that for a constant $C>0$,
\begin{align*}
|x(t)-y_{\epsilon}(t)|\leq C\left(|\sigma_{{\epsilon}}(t)|+|\lambda_{{\epsilon}}(t)|+|\gamma_{{\epsilon}}(t)|\right)+\mathcal{O}(\epsilon^{2})
\end{align*}
for every $t\in [0, T_{\epsilon}^{\ast}]$.
\end{lemma}
\begin{proof}
First, we claim that there exists $T_{0}>0$ such that $|y_{\epsilon}(t)|<\rho$, for every $t\in [0, T_{\epsilon}^{\ast})$ with $T_{\epsilon}^{\ast}\leq T_{0}$, where the constant $\rho$ is defined in \eqref{cp}. Let us first prove that
\begin{align*}
\|\delta_{y_{\epsilon} (t_{1})}-\delta_{y_{\epsilon}(t_{2})}\|_{C^{2\ast}}<\rho, \quad \text{for all $t_{1}$, $t_{2}\in [0, T_{\epsilon}^{\ast}]$.}
\end{align*}
Let $f\in C^{2}({\mathbb R}^{N})$ with $\|f\|_{C^{2}}\leq 1$ and pick $t_{1}$, $t_{2}\in [0, T_{\epsilon}^{\ast})$ . From Lemma \ref{L1} and identity \eqref{Id1} we see that
\begin{align*}
\frac{1}{\epsilon^{N}}\int_{{\mathbb R}^{N}}&\left(|u_{\epsilon}(x, t_{2})|^{2}-|u_{\epsilon}(x, t_{1})|^{2}\right)f(x)\, dx=\frac{1}{\epsilon^{N}}\int_{{\mathbb R}^{N}}\int^{t_{2}}_{t_{1}}\frac{\partial |u_{\epsilon}|^{2}}{\partial t}(x,t)f(x)\,dtdx\\
&=\int_{{\mathbb R}^{N}}\int^{t_{2}}_{t_{1}}-f(x)\mbox{div}_{x}p_{\epsilon}(x,t)dtdx=\int^{t_{2}}_{t_{1}}\int_{{\mathbb R}^{N}}\nabla f(x)\cdot p_{\epsilon}(x,t)\,dxdt\\
&\leq \|\nabla f\|_{L^{\infty}}\int^{t_{2}}_{t_{1}}dt\int_{{\mathbb R}^{N}}|p_{\epsilon}(x,t)|dx\leq C|t_{2}-t_{1}|.
\end{align*}
Therefore, there exists a constant $C>0$ such that
\begin{align*}
\left\|\frac{|u_{\epsilon}(x, t_{2})|^{2} }{\epsilon^{N}}dx-\frac{|u_{\epsilon}(x, t_{1})|^{2} }{\epsilon^{N}}dx\right\|_{C^{2\ast}}\leq C|t_{2}-t_{1}|\leq CT_{0}.
\end{align*}
Now, from Lemma \ref{L23} we obtain
\begin{align*}
m\|\delta_{y_{\epsilon}(t_{2})}-\delta_{y_{\epsilon}(t_{1})}\|_{C^{2\ast}}\leq CT_{0}+C|\sigma_{{\epsilon}}(t)|+C|\lambda_{{\epsilon}}(t)|+\mathcal{O}(\epsilon^{2})\leq C(T_{0}+h/2)+\mathcal{O}(\epsilon^{2}).
\end{align*}
Here we choose $T_{0}$ and then $\epsilon_{0}$, $h_{0}$ such that $ C(T_{0}+h/2)+\mathcal{O}(\epsilon^{2})<\min\left\{m K_{0}, mK_{1}K_{0}\right\}$, where $K_{0}$ and $K_{1}$ are  the constants defined in formula \eqref{Dp}. Thus, from inequality \eqref{Dp} we get $|y_{\epsilon}(t_{2})-y_{\epsilon}(t_{1})|<K_{0}$ for every $t_{1}$, $t_{2}\in [0, T_{\epsilon}^{\ast})$, and since $y_{\epsilon}(0)=0$,  this implies the claim.
We now conclude the proof of lemma.  Since the definition of $\chi$, it follows that 
\begin{align*}
|x(t)-y_{\epsilon}(t)|&=\frac{1}{m}|mx(t)-my_{\epsilon}(t)|\\
&\leq \frac{1}{m}|\gamma_{\epsilon}(t)|+\frac{1}{m}\left|\int_{{\mathbb R}^{N}}x\chi(x)\frac{|u_{\epsilon}(x,t)|^{2}}{\epsilon^{N}}dx-my_{\epsilon}(t)\right|.
\end{align*}
Notice that from claim above and \eqref{cha} we see that $\chi(y_{\epsilon}(t))=1$ for all $t\in [0, T_{\epsilon}^{\ast})$. In particular, there exists a constant $C>0$ such that
\begin{align*}
|x(t)-y_{\epsilon}(t)|\leq C \|x\chi\|_{C^{2}}\left\|\frac{|u_{\epsilon}(x,t)|^{2}}{\epsilon^{N}}dx-m\delta_{y\epsilon}(t)\right\|_{C^{2\ast}}+C|\gamma_{\epsilon}(t)|.
\end{align*}
Then the statement follows by Lemma \ref{L23}.
\end{proof}

Using Lemmas \ref{L23} and \ref{L26} and inequality \eqref{Dp}, one can prove the following result.
\begin{lemma} \label{L29} 
There exists a positive constant $C$ such that
\begin{align*}
&\left\|\frac{|u_{\epsilon}(x,t)|^{2}}{\epsilon^{N}}dx-m\delta_{x(t)}\right\|_{(C^{2})^{\ast}}+\left\|p_{\epsilon}(x,t)dx-m\nu(t)\delta_{x(t)}\right\|_{(C^{2})^{\ast}} 
\leq C\Xi(t)+\mathcal{O}(\epsilon^{2}),
\end{align*}
where $\Xi(t):=|\sigma_{{\epsilon}}(t)|+|\lambda_{{\epsilon}}(t)|+ |\gamma_{{\epsilon}}(t)|$,  for all $\epsilon\in (0, \epsilon_{0})$ and $t\in [0, T^{\ast}_{\epsilon})$
\end{lemma}
\begin{proof}
The proof follows the same lines as Lemma 6.4 in \cite{SQ1}.
\end{proof}

In Lemma \ref{L26} we have fixed $T_{0}$ such that Proposition \ref{PA12} and Lemmas \ref{L23} and \ref{L29} hold. With this in mind,
now we give the proof of Theorem \ref{T2}.

\begin{proof}[\bf{Proof of Theorem \ref{T2}}] 
The proof follows the lines of the corresponding proof in \cite[Theorem 1.1 and Lemma 3.6]{SK2}. Let us give a brief sketch of the proof. First, we want to use a Gronwall inequality argument to show that 
\begin{align}\label{Coo}
|\sigma_{\epsilon}(t)|+|\lambda_{\epsilon}(t)|+|\gamma_{\epsilon}(t)|\leq C(T_{0})\epsilon^{2}\quad \text{for every $t\in [0, T^{\ast}_{\epsilon}]$}.
\end{align}
Indeed, from identities \eqref{Id1}, \eqref{Id2} and Lemmas \ref{L26} and \ref{L29},  and repeating the steps of the proof of Lemma 3.6 in \cite{SK2} we get for every $t\in [0, T^{\ast}_{\epsilon}]$,
\begin{align*}
\left|\frac{{d}}{{d}t}\sigma_{\epsilon}(t)\right|+\left|\frac{{d}}{{d}t}\lambda_{\epsilon}(t)\right|+\left|\frac{{d}}{{d}t}\gamma_{\epsilon}(t)\right|\leq C\left[|\sigma_{\epsilon}(t)|+|\lambda_{\epsilon}(t)|+|\gamma_{\epsilon}(t)|+\mathcal{O}(\epsilon^{2})\right],
\end{align*}
for some positive constant $C$. Moreover, by Lemmas \ref{L8} and \ref{L25} we see that $|\sigma_{\epsilon}(0)|+|\lambda_{\epsilon}(0)|+|\gamma_{\epsilon}(0)|=\mathcal{O}(\epsilon^{2})$, then \eqref{Coo} is a simple consequence of the Gronwall inequality. By the definition of $T^{\ast}_{\epsilon}$  in formula \eqref{Tim} and  due to the continuity of  $\sigma_{\epsilon}$, $\lambda_{\epsilon}$ and $\gamma_{\epsilon}$  one gets $T^{\ast}_{\epsilon}=T_{0}$  for $\epsilon$ small enough, $\epsilon\in (0, \epsilon_{0})$. Next, in light of Proposition \ref{PA12} there exist families of bounded functions $\theta_{\epsilon}: \mathbb{R}\rightarrow [0, 2\pi)$, $y_{\epsilon}:\mathbb{R}\rightarrow \mathbb{R}^{N}$ such that 
\begin{equation*}
\left\|u_{\epsilon}(\cdot, t)-e^{\frac{1}{\epsilon}(\nu(t)\cdot x+\theta_{\epsilon}(t))}R\left(\frac{x-y_{\epsilon}(t)}{\epsilon}\right)\right\|^{2}_{H^{1}_{\epsilon}}=\mathcal{O}(\epsilon^{2}),
\end{equation*}
for all $t\in [0, T_{0}]$. Furthermore, from Lemma \ref{L26} and \eqref{Coo}, it is clear that $|x(t)-y_{\epsilon}(t)|\leq C\epsilon^{2}$ for $t\in [0, T_{0}]$ and $\epsilon\in (0, \epsilon_{0})$. Therefore,
\begin{equation*}
\left\|R\left(\frac{x-y_{\epsilon}(t)}{\epsilon}\right)-R\left(\frac{x-x(t)}{\epsilon}\right)\right\|^{2}_{H^{1}_{\epsilon}}\leq C\frac{|x(t)-y_{\epsilon}(t)|^{2}}{\epsilon^{2}}=\mathcal{O}(\epsilon^{2}),
\end{equation*}
for every $t\in [0, T_{0}]$ and $\epsilon\in (0, \epsilon_{0})$. Hence the Theorem \ref{T2} holds in $[0, T_{0}]$. Finally, taking as new data
$x(T_{0})$ and $\nu(T_{0})$ in system \eqref{HS} and 
\begin{equation*}
u_{\epsilon, 0}(x):=e^{\frac{i}{\epsilon}x\cdot \nu(T_{0})}e^{\frac{1+N}{2}}e^{-\frac{1}{\epsilon^{2}}|x-x(T_{0})|^{2}},
\end{equation*}
as a new initial data in Cauchy problem \eqref{0NL}, the statement is valid on time $[T_{0}, 2T_{0}]$. Since $T_{0}$ only depends on the problem, we can achieve any finite time interval [0, T].  This concludes the proof of Theorem \ref{T2}.
\end{proof}

\bigskip

\end{document}